\documentclass[12pt,reqno,draft]{amsart}

\usepackage{amsmath,amssymb,amscd,latexsym,amsthm,bm}
\usepackage{times}

\newtheorem{theorem}{Theorem}[section]

\newtheorem{lemma}[theorem]{Lemma}
\newtheorem{remark}[theorem]{Remark}

\numberwithin{equation}{section}

\begin{document}

\title[On period polynomials of degree $2^m$]{On period polynomials of degree $\bm{2^m}$\\  for finite fields}

\author{Ioulia N. Baoulina}

\address{Department of Mathematics, Moscow State Pedagogical University, Krasnoprudnaya str. 14, Moscow 107140, Russia}
\email{jbaulina@mail.ru}

\date{}

\maketitle

\begin{abstract}
We obtain explicit factorizations of reduced period polynomials of degree $2^m$, $m\ge 4$, for finite fields of characteristic $p\equiv 3\text{\;or\;}5\pmod{8}$. This extends the results of G.~Myerson, who considered the cases $m=1$ and $m=2$, and S.~Gurak, who studied the case $m=3$.
\end{abstract}



\thispagestyle{empty}

\section{Introduction}
Let $\mathbb F_q$ be a finite field of characteristic~$p$ with $q=p^s$ elements, $\mathbb F_q^*=\mathbb F_q^{}\setminus\{0\}$, and let $\gamma$ be a fixed generator of the cyclic group $\mathbb F_q^*$ . By ${\mathop{\rm Tr}\nolimits}:\mathbb F_q\rightarrow\mathbb F_p$ we denote the trace mapping, that is,  ${\mathop{\rm
Tr}\nolimits}(x)=x+x^p+x^{p^2}+\dots+x^{p^{s-1}}$ for $x\in\mathbb F_q$. Let $e$ and $f$ be positive integers such that $q=ef+1$. Denote by $\mathcal{H}$ the subgroup of $e$-th powers in $\mathbb F_q^*$.  For any positive integer $n$, write $\zeta_n=\exp(2\pi i/n)$.

The cyclotomic (or $f$-nomial) periods of order $e$ for $\mathbb F_q$ with respect to $\gamma$  are defined by
$$
\eta_k=\sum_{x\in\gamma^k\mathcal{H}}\zeta_p^{{\mathop{\rm Tr}\nolimits}(x)}=\sum_{h=0}^{f-1}\zeta_p^{{\mathop{\rm Tr}\nolimits}(\gamma^{eh+k})},\quad k=0,1,\dots,e-1.
$$
The period polynomial of degree $e$ for $\mathbb F_q$ is the polynomial
$$
P_e(X)=\prod_{k=0}^{e-1}(X-\eta_k).
$$
The reduced cyclotomic (or reduced $f$-nomial) periods of order $e$ for $\mathbb F_q$ with respect to $\gamma$ are defined by
$$
\eta_k^*=\sum_{x\in\mathbb F_q}\zeta_p^{{\mathop{\rm Tr}\nolimits}(\gamma^k x^e)}=1+e\eta_k,\quad k=0,1,\dots,e-1,
$$
and the reduced period polynomial of degree $e$ for $\mathbb F_q$ is
$$
P_e^*(X)=\prod_{k=0}^{e-1}(X-\eta_k^*).
$$

The polynomials $P_e(X)$ and $P_e^*(X)$ have integer coefficients and  are independent of the choice of generator~$\gamma$. They are irreducible over the rationals when $s=1,$  but not necessarily irreducible when $s>1$. More precisely, $P_e(X)$ and $P_e^*(X)$ split over the rationals into $\delta=\gcd(e,(q-1)/(p-1))$ factors of degree~$e/\delta$ (not necessarily distinct), and each of these factors is irreducible or a power of an irreducible polynomial. Furthermore, the polynomials $P_e(X)$ and $P_e^*(X)$ are irreducible over the rationals if and only if $\gcd(e,(q-1)/(p-1))=1$. For proofs of these facts, see~\cite{M}.

In the case $s=1$, the period polynomials were determined explicitly by Gauss for ${e\in\{2, 3, 4\}}$  and by many others for certain small values of~$e$. In the general case, Myerson~\cite{M} derived the explicit formulas for $P_e(X)$ and $P_e^*(X)$ when $e\in\{2,3,4\}$, and also found their factorizations into irreducible polynomials over the rationals. Gurak~\cite{G3} obtained similar results for $e\in\{6,8,12,24\}$; see also \cite{G2} for the case $s=2$, $e\in\{6,8,12\}$. Note that if $-1$ is a power of $p$ modulo $e$, then the period polynomials can also be easily obtained. Indeed, if $e>2$ and $e\mid(p^{\ell}+1)$, with $\ell$ chosen minimal, then $2\ell\mid s$, and \cite[Proposition~20]{M} yields
$$
P_e^*(X)=(X+(-1)^{s/2\ell}(e-1)q^{1/2})(X-(-1)^{s/2\ell}q^{1/2})^{e-1}.
$$
Baumert and Mykkeltveit~\cite{BM} found the values of cyclotomic periods in the case when $e>3$ is a prime, $e\equiv 3\pmod{4}$  and $p$ generates the quadratic residues modulo~$e$; see also \cite[Proposition~21]{M}.

It is seen immediately from the definitions that $P_e(X)=e^{-e}P_e^*(eX+1)$, and so it suffices to factorize only $P_e^*(X)$.

The aim of this paper is to obtain the explicit factorizations of the reduced period polynomials of degree $2^m$ with $m\ge 4$ in the case that $p\equiv 3\text{\;or\;}5\pmod{8}$. Notice that in this case $\mathop{\rm ord}_2(q-1)=\mathop{\rm ord}_2(p^s-1)=\mathop{\rm ord}_2 s+2$. Hence, for $p\equiv 3\pmod{8}$,
$$
\gcd(2^m,(q-1)/(p-1))=\begin{cases}
2^m&\text{if $2^{m-1}\mid s$,}\\
2^{m-1}&\text{if $2^{m-2}\parallel s$.}
\end{cases}
$$
Appealing to \cite[Theorem~4]{M}, we conclude that in the case when $2^{m-1}\mid s$, $P_{2^m}^*(X)$ splits over the rationals into linear factors. If $2^{m-2}\parallel s$, then $P_{2^m}^*(X)$ splits into irreducible polynomials of degrees at most 2. Similarly, for $p\equiv 5\pmod{8}$,
$$
\gcd(2^m,(q-1)/(p-1))=\begin{cases}
2^m&\text{if $2^m\mid s$,}\\
2^{m-1}&\text{if $2^{m-1}\parallel s$,}\\
2^{m-2}&\text{if $2^{m-2}\parallel s$.}
\end{cases}
$$
Using \cite[Theorem~4]{M} again, we see that $P_{2^m}^*(X)$ splits over the rationals into linear factors if $2^m\mid s$, splits into linear and quadratic irreducible factors if $2^{m-1}\parallel s$, and splits into linear, quadratic and biquadratic irreducible factors if $2^{m-2}\parallel s$. Our main results are Theorems~\ref{t1} and \ref{t2}, which give the explicit factorizations of $P_{2^m}^*(X)$ in the cases $p\equiv 3\pmod{8}$ and $p\equiv 5\pmod{8}$, respectively. All the evaluations in Sections~\ref{s3} and \ref{s4} are effected in terms of parameters occuring in quadratic partitions of some powers of~$p$.

\section{Preliminary Lemmas}
\label{s2}

In the remainder of the paper, we assume that $p$ is an odd prime. Let $\psi$ be a nontrivial character on $\mathbb F_q$. We extend $\psi$ to all of $\mathbb F_q$ by setting $\psi(0)=0$. The Gauss sum
$G(\psi)$ over $\mathbb F_q$ is defined by
$$
G(\psi)=\sum_{x\in\mathbb F_q}\psi(x)\zeta_p^{{\mathop{\rm
Tr}\nolimits}(x)}.
$$
Gauss sums occur in the Fourier expansion of a reduced cyclotomic period.
\begin{lemma}
\label{l1}
Let $\psi$ be a character of order $e>1$ on $\mathbb F_q$ such that $\psi(\gamma)=\zeta_e$. Then for $k=0,1,\dots, e-1$,
$$
\eta_k^*=\sum_{j=1}^{e-1} G(\psi^j)\zeta_e^{-jk}.
$$
\end{lemma}

\begin{proof}
It follows from \cite[Theorem~1.1.3 and Equation~(1.1.4)]{BEW}.
\end{proof}

In the next three lemmas, we record some properties of Gauss sums which will be used throughout this paper. By $\rho$ we denote the quadratic character on $\mathbb F_q$  ($\rho(x)=+1, -1, 0$ according as $x$ is a square, a non-square or zero in $\mathbb F_q$).
\begin{lemma}
\label{l2}
Let $\psi$ be a nontrivial character on $\mathbb F_q$
with $\psi\ne\rho$. Then
\begin{itemize}
\item[\textup{(a)}] 
$G(\psi)G(\bar\psi)=\psi(-1)q$;
\item[\textup{(b)}] 
$G(\psi)=G(\psi^p)$;
\item[\textup{(c)}] 
$G(\psi)G(\psi\rho)=\bar\psi(4)G(\psi^2)G(\rho)$.
\end{itemize}
\end{lemma}

\begin{proof}
See \cite[Theorems~1.1.4(a, d) and 11.3.5]{BEW} or \cite[Theorem~5.12(iv, v) and Corollary~5.29]{LN}.
\end{proof}

\begin{lemma}
\label{l3}
We have
$$
G(\rho)=
\begin{cases}
(-1)^{s-1}q^{1/2}&\text{if\,\, $p\equiv 1\pmod{4}$,}\\
(-1)^{s-1}i^s q^{1/2}&\text{if\,\, $p\equiv 3\pmod 4$.}
\end{cases}
$$ 
\end{lemma}

\begin{proof}
See \cite[Theorem~11.5.4]{BEW} or \cite[Theorem~5.15]{LN}.
\end{proof}

\begin{lemma}
\label{l4}
Let $p\equiv 3\pmod{8}$, $2\mid s$ and $\psi$ be a biquadratic character on $\mathbb F_q$. Then $G(\psi)=-q^{1/2}$.
\end{lemma}

\begin{proof}
It is a special case of  \cite[Theorem~11.6.3]{BEW}.
\end{proof}

Let $\psi$ be a nontrivial character on $\mathbb F_q$. The Jacobi sum $J(\psi)$ over $\mathbb F_q$ is defined by
$$
J(\psi)=\sum_{x\in\mathbb F_q}\psi(x)\psi(1-x).
$$
The following lemma gives a relationship between Gauss sums and Jacobi sums.

\begin{lemma}
\label{l5}
Let $\psi$ be a nontrivial character on $\mathbb F_q$
with $\psi\ne\rho$. Then
$$
G(\psi)^2=G(\psi^2)J(\psi).
$$
\end{lemma}

\begin{proof}
See \cite[Theorem~2.1.3(a)]{BEW} or \cite[Theorem~5.21]{LN}.
\end{proof}

Let $\psi$ be a character on $\mathbb F_q$. The lift $\psi'$ of
the character $\psi$ from $\mathbb F_{q^{\vphantom{r}}}$ to the
extension field $\mathbb F_{q^r}$ is given by
$$
\psi'(x)=\psi({\mathop{\rm N}}_{\mathbb F_{q^r}/\mathbb
F_{q^{\vphantom{r}}}}(x)), \qquad x\in\mathbb F_{q^r},
$$
where ${\mathop{\rm N}}_{\mathbb F_{q^r}/\mathbb
F_{q^{\vphantom{r}}}}(x)=x\cdot x^q\cdot x^{q^2}\cdots
x^{q^{r-1}}=x^{(q^r-1)/(q-1)}$ is the norm of $x$ from
$\mathbb F_{q^r}$ to $\mathbb F_{q^{\vphantom{r}}}$.

\begin{lemma}
\label{l6} 
Let $\psi$ be a character on $\mathbb
F_{q^{\vphantom{r}}}$ and let $\psi'$ denote the lift of $\psi$
from $\mathbb F_{q^{\vphantom{r}}}$ to $\mathbb F_{q^r}$. Then
\begin{itemize}
\item[\textup{(a)}] 
$\psi'$ is a character on $\mathbb F_{q^r}$;
\item[\textup{(b)}]
a character $\lambda$ on $\mathbb F_{q^r}$ equals the lift $\psi'$ of some character $\psi$ on $\mathbb F_q$ if and only if the order of $\lambda$ divides $q-1$;
\item[\textup{(c)}]
$\psi'$ and $\psi$ have the same order.
\end{itemize}
\end{lemma}

\begin{proof}
See \cite[Theorem~11.4.4(a, c, e)]{BEW}.
\end{proof}

The following lemma, which is due to Davenport and Hasse, connects a Gauss sum and its lift. 
\begin{lemma}
\label{l7} 
Let $\psi$ be a nontrivial character on $\mathbb F_q$
and let $\psi'$ denote the lift of $\psi$ from $\mathbb F_{q^{}}$
to $\mathbb F_{q^r}$. Then
$$
G(\psi')=(-1)^{r-1}G(\psi)^r.
$$
\end{lemma}

\begin{proof}
See \cite[Theorem~11.5.2]{BEW} or \cite[Theorem~5.14]{LN}.
\end{proof}

Now we turn to the case $p\equiv 3\text{\;or\;}5\pmod{8}$. We recall a few facts which were established in our earlier paper~\cite{B2} in more general settings.

\begin{lemma}
\label{l8} 
Let $p\equiv 3\text{\;or\;}5\pmod{8}$ and
$\psi$ be a character of order~$2^r$ on $\mathbb F_q$, where 
$$
r\ge\begin{cases}
4&\text{if $p\equiv 3\pmod{8}$,}\\
3&\text{if $p\equiv 5\pmod{8}$.}
\end{cases}
$$ 
Then $G(\psi)=G(\psi\rho)$.
\end{lemma}

\begin{proof}
See \cite[Lemma 2.13]{B2}.
\end{proof}

\begin{lemma}
\label{l9} 
Let $p\equiv 3\text{\;or\;}5\pmod{8}$, $r\ge 3$, and
$\psi$ be a character of order~$2^r$ on $\mathbb F_q$. Then
$$
\psi(4)=
\begin{cases}
1 & \text{if $p\equiv 3\pmod{8}$,}\\
(-1)^{s/2^{r-2}}&\text{if $p\equiv 5\pmod{8}$.}
\end{cases}
$$
\end{lemma}

\begin{proof}
See \cite[Lemma 2.16]{B2}.
\end{proof}

\begin{lemma}
\label{l10}
Let $p\equiv 3\text{\;or\;}5\pmod{8}$, $n\ge 1$ and $r\ge 3$ be integers, $r\ge n$. Then
$$
\sum_{v=0}^{2^{r-2}-1}\zeta_{2^n}^{p^v}=\begin{cases}
-2^{r-2}&\text{if $n=1$,}\\
2^{r-2}i&\text{if $n=2$ and $p\equiv 5\pmod{8}$,}\\
2^{r-3}i\sqrt{2}&\text{if $n=3$ and $p\equiv 3\pmod{8}$,}\\
0&\text{otherwise.}
\end{cases}
$$
\end{lemma}

\begin{proof}
It is an immediate consequence of \cite[Lemma 2.2]{B2}.
\end{proof}

 The next lemma relates Gauss sums over $\mathbb F_q$ to Jacobi sums over a subfield of $\mathbb F_q$.

\begin{lemma}
\label{l11}
Let $p\equiv 3\text{\;or\;}5\pmod{8}$, and $\psi$ be a character of order $2^r$  on $\mathbb F_q$, where
$$
r\ge n=\begin{cases}
3&\text{if $p\equiv 3\pmod{8}$,}\\
2&\text{if $p\equiv 5\pmod{8}$,}
\end{cases}
$$
Assume that $2^{r-1}\mid s$. Then $\psi^{2^{r-n}}$  is equal to the lift of some character $\chi$ of order~$2^n$ on $\mathbb F_{p^{s/2^{r-n+1}}}$. Moreover,  
$$
G(\psi)=q^{(2^{r-n+1}-1)/2^{r-n+2}}J(\chi)\cdot\begin{cases}
1&\text{if $p\equiv 3\pmod{8}$,}\\
(-1)^{s(r-1)/2^{r-1}}&\text{if $p\equiv 5\pmod{8}$.}
\end{cases}
$$
\end{lemma}

\begin{proof}
We prove the assertion of the lemma by induction on $r$, for $r\ge n$. Let $2^{n-1}\mid s$ and $\psi$ be a character of order $2^n$ on $\mathbb F_q$. As $2^n\mid(p^{s/2}-1)$, Lemma~\ref{l6} shows that $\psi$ is equal to the lift of some character $\chi$ of order $2^n$ on $\mathbb F_{p^{s/2}}$, that is, $\chi'=\psi$. Lemmas~\ref{l5} and \ref{l7} yield $G(\psi)=G(\chi')=-G(\chi)^2=-G(\chi^2)J(\chi)$. Note that $\chi^2$ has order~$2^{n-1}$. Thus, by Lemmas~\ref{l3} and \ref{l4}, 
$$
G(\chi^2)=\begin{cases}
-q^{1/4}&\text{if $p\equiv 3\pmod{8}$,}\\
(-1)^{(s/2)-1}q^{1/4}&\text{if $p\equiv 5\pmod{8}$,}
\end{cases}
$$
and so 
$$
G(\psi)=q^{1/4}J(\chi)\cdot\begin{cases}
1&\text{if $p\equiv 3\pmod{8}$,}\\
(-1)^{s/2}&\text{if $p\equiv 5\pmod{8}$.}
\end{cases}
$$
This completes the proof for the case $r=n$.

Suppose now that $r\ge n+1$, and assume that the result is true when $r$ is replaced by $r-1$. Let $2^{r-1}\mid s$ and $\psi$ be a character of order $2^r$  on $\mathbb F_q$. Then $2^{r-2}\mid\frac s2$, and so $2^r\mid(p^{s/2}-1)$. By Lemma~\ref{l6}, $\psi$ is equal to the lift of some character $\phi$ of order $2^r$ on $\mathbb F_{p^{s/2}}$, that is $\phi'=\psi$. Applying Lemmas \ref{l2}(c), \ref{l3}, \ref{l7}, \ref{l8} and using the fact that $2^n\mid s$, we deduce
\begin{equation}
\label{eq1}
G(\psi)=-G(\phi)^2=-G(\phi)G(\phi\rho_0)=-\bar\phi(4)G(\phi^2)G(\rho_0)=\bar\phi(4)q^{1/4}G(\phi^2),
\end{equation}
where $\rho_0$ denotes the quadratic character on $\mathbb F_{p^{s/2}}$. Note that $\phi^2$ has order $2^{r-1}$ and $2^{r-2}\mid\frac s2$. Hence, by inductive hypothesis, $(\phi^2)^{2^{r-1-n}}=\phi^{2^{r-n}}$ is equal to the lift of some character $\chi$ of order $2^n$ on $\mathbb F_{p^{(s/2)/2^{r-n}}}=\mathbb F_{p^{s/2^{r-n+1}}}$ and 
$$
G(\phi^2)=(p^{s/2})^{(2^{r-n}-1)/2^{r-n+1}}J(\chi)\cdot\begin{cases}
1&\text{if $p\equiv 3\pmod{8}$,}\\
(-1)^{(s/2)(r-2)/2^{r-2}}&\text{if $p\equiv 5\pmod{8}$,}
\end{cases}
$$
that is,
$$
G(\phi^2)=q^{(2^{r-n}-1)/2^{r-n+2}}J(\chi)\cdot\begin{cases}
1&\text{if $p\equiv 3\pmod{8}$,}\\
(-1)^{s(r-2)/2^{r-1}}&\text{if $p\equiv 5\pmod{8}$.}
\end{cases}
$$
Substituting this expression for $G(\phi^2)$ into \eqref{eq1} and using Lemma~\ref{l9}, we obtain
$$
G(\psi)=q^{(2^{r-n+1}-1)/2^{r-n+2}}J(\chi)\cdot\begin{cases}
1&\text{if $p\equiv 3\pmod{8}$,}\\
(-1)^{s(r-1)/2^{r-1}}&\text{if $p\equiv 5\pmod{8}$.}
\end{cases}
$$
It remains to show that $\psi^{2^{r-n}}$  is equal to the lift of $\chi$. Indeed, for any $x\in\mathbb F_q$ we have
\begin{align*}
\chi({\mathop{\rm N}}_{\mathbb F_q/\mathbb F_{p^{s/2^{r-n+1}}}}(x))&=\chi(x^{(p^s-1)/(p^{s/2^{r-n+1}}-1)})\\
&=\chi((x^{(p^s-1)/(p^{s/2}-1)})^{(p^{s/2}-1)/(p^{s/2^{r-n+1}}-1)})\\
&=\chi({\mathop{\rm N}}_{\mathbb F_{p^{s/2}}/\mathbb F_{p^{s/2^{r-n+1}}}}(x^{(p^s-1)/(p^{s/2}-1)}))=\phi^{2^{r-n}}(x^{(p^s-1)/(p^{s/2}-1)})\\
&=\left(\phi({\mathop{\rm N}}_{\mathbb F_{p^s}/\mathbb F_{p^{s/2}}}(x))\right)^{2^{r-n}}=\psi^{2^{r-n}}(x).
\end{align*}
Therefore $\chi'=\psi^{2^{r-n}}$, and the result now follows by the principle of mathematical induction.
\end{proof}

For an arbitrary integer $k$, it is convenient to set $\eta_k^*=\eta_{\ell}^*$, where $k\equiv {\ell}\pmod{e}$, $0\le\ell\le e-1$.

\begin{lemma}
\label{l12}
For any integer $k$, $\eta_{kp}^*=\eta_k^*$.
\end{lemma}

\begin{proof}
It is a straightforward consequence of \cite[Proposition~1]{G1}.
\end{proof}

From now on we shall assume that $p\equiv 3\text{\;or\;}5\pmod{8}$, $e=2^m$ with $m\ge 3$, and $\lambda$ is a character of order $2^m$ on $\mathbb F_q$ such that $\lambda(\gamma)=\zeta_{2^m}$. We observe that $2^{m-2}\mid s$. 

\begin{lemma}
\label{l13}
We have
$$
P_{2^m}^*(X)=(X-\eta_0^*)(X-\eta_{2^{m-1}}^*)\prod_{t=0}^{m-2}(X-\eta_{2^t}^*)^{2^{m-t-2}}(X-\eta_{-2^t}^*)^{2^{m-t-2}}.
$$
\end{lemma}

\begin{proof}
Write
\begin{align*}
P_{2^m}^*(X)&=(X-\eta_0^*)(X-\eta_{2^{m-1}}^*)\prod_{t=0}^{m-2}\,\prod_{\substack{k=1\\ 2^t\parallel k}}^{2^m-1}(X-\eta_k^*)\\
&=(X-\eta_0^*)(X-\eta_{2^{m-1}}^*)\prod_{t=0}^{m-2}\,\prod_{\substack{k_0=1\\ 2\nmid k_0}}^{2^{m-t}-1}(X-\eta_{2^tk_0}^*).
\end{align*}
Since $p\equiv 3\text{\;or\;}5\pmod{8}$, \,$\pm p^0,\pm p^1,\dots, \pm p^{2^{m-t-2}-1}$ is a reduced residue system modulo $2^{m-t}$ for each $0\le t\le m-2$. Thus
$$
P_{2^m}^*(X)=(X-\eta_0^*)(X-\eta_{2^{m-1}}^*)\prod_{t=0}^{m-2}\,\prod_{j=0}^{2^{m-t-2}-1}(X-\eta_{2^t p^j}^*)(X-\eta_{-2^t p^j}^*).
$$
The result now follows from Lemma~\ref{l12}.
\end{proof}

\begin{lemma}
\label{l14}
We have
\begin{align*}
\eta_0^*=\,&G(\rho)+\sum_{r=2}^m 2^{r-2}\left(G(\lambda^{2^{m-r}})+G(\bar\lambda^{2^{m-r}})\right),\\
\eta_{2^{m-1}}^*=\,&G(\rho)+\sum_{r=2}^{m-1} 2^{r-2}\left(G(\lambda^{2^{m-r}})+G(\bar\lambda^{2^{m-r}})\right)-2^{m-2}\left(G(\lambda)+G(\bar\lambda)\right),
\end{align*}
and, for $0\le t\le m-2$,
\begin{align*}
\eta_{\pm 2^t}^*=\,&\sum_{r=2}^t 2^{r-2}\left(G(\lambda^{2^{m-r}})+G(\bar\lambda^{2^{m-r}})\right),\\
&+\begin{cases}
-G(\rho)&\text{if $t=0$,}\\
G(\rho)-2^{t-1}\left(G(\lambda^{2^{m-t-1}})+G(\bar\lambda^{2^{m-t-1}})\right)&\text{if $t>0$,}
\end{cases}\\
&\mp\begin{cases}
0&\text{if $p\equiv 3\pmod{8}$,}\\
2^ti\,\left(G(\lambda^{2^{m-t-2}})-G(\bar\lambda^{2^{m-t-2}})\right)&\text{if $p\equiv 5\pmod{8}$,}
\end{cases}\\
&\mp\begin{cases}
2^ti\sqrt{2}\,\left(G(\lambda^{2^{m-t-3}})-G(\bar\lambda^{2^{m-t-3}})\right)&\text{if $p\equiv 3\pmod{8}$ and $t\le m-3$,}\\
0&\text{otherwise.}
\end{cases}
\end{align*}
\end{lemma}

\begin{proof}
From Lemma~\ref{l1} we deduce that
$$
\eta_k^*=\sum_{j=1}^{2^m-1}G(\lambda^j)\zeta_{2^m}^{-jk}=\sum_{r=1}^m \sum_{\substack{j=1\\ 2^{m-r}\parallel j}}^{2^m-1}G(\lambda^j)\zeta_{2^m}^{-jk}
=\sum_{r=1}^m \sum_{\substack{j_0=1\\ 2\nmid j_0}}^{2^r-1}G(\lambda^{2^{m-r}j_0})\zeta_{2^r}^{-j_0k}.
$$
Since $\lambda^{2^{m-r}}$ has order $2^r$ and, for $r\ge 2$, $\pm p^0,\pm p^1,\dots,\pm p^{2^{r-2}-1}$ is a reduced residue system modulo $2^r$, we conclude that
$$
\eta_k^*=(-1)^k G(\rho)+\sum_{r=2}^m\, \sum_{u\in\{\pm 1\}}\sum_{v=0}^{2^{r-2}-1}G(\lambda^{2^{m-r}up^v})\zeta_{2^r}^{-kup^v},
$$
or, in view of Lemma~\ref{l2}(b),
\begin{equation}
\label{eq2}
\eta_k^*=(-1)^k G(\rho)+\sum_{r=2}^m\left[G(\lambda^{2^{m-r}})\sum_{v=0}^{2^{r-2}-1}\zeta_{2^r}^{-kp^v}+G(\bar\lambda^{2^{m-r}})\sum_{v=0}^{2^{r-2}-1}\zeta_{2^r}^{kp^v}\right].
\end{equation}
The expressions for $\eta_0^*$ and $\eta_{2^{m-1}}^*$ follow immediately from \eqref{eq2}. Next we assume that $0\le t\le m-2$. If $r>t+3$, then, by Lemma~\ref{l10},
$$
\sum_{v=0}^{2^{r-2}-1}\zeta_{2^r}^{2^t p^v}=\sum_{v=0}^{2^{r-2}-1}\zeta_{2^r}^{-2^t p^v}=0,
$$
and so \eqref{eq2} yields
\begin{align*}
\eta_{2^t}^*=\,&\sum_{r=2}^t 2^{r-2}\left(G(\lambda^{2^{m-r}})+G(\bar\lambda^{2^{m-r}})\right),\\
&+\begin{cases}
-G(\rho)&\text{if $t=0$,}\\
G(\rho)-2^{t-1}\left(G(\lambda^{2^{m-t-1}})+G(\bar\lambda^{2^{m-t-1}})\right)&\text{if $t>0$,}
\end{cases}\\
&+G(\lambda^{2^{m-t-2}})\sum_{v=0}^{2^t-1}i^{-p^v}+G(\bar\lambda^{2^{m-t-2}})\sum_{v=0}^{2^t-1}i^{p^v}\\
&+\begin{cases}
G(\lambda^{2^{m-t-3}})\sum_{v=0}^{2^{t+1}-1}\zeta_8^{-p^v}+G(\bar\lambda^{2^{m-t-3}})\sum_{v=0}^{2^{t+1}-1}\zeta_8^{p^v}&\text{if $t\le m-3$,}\\
0&\text{if $t=m-2$.}
\end{cases}
\end{align*}
The asserted result now follows from Lemmas~\ref{l4} and \ref{l10}. The expression for  $\eta_{-2^t}^*$  can be obtained in a similar manner.
\end{proof}

\section{The Case $p\equiv 3\pmod{8}$}
\label{s3}

In this section, $p\equiv 3\pmod{8}$. As before, $2^m\mid(q-1)$ and $\lambda$ is a character of order~$2^m$ on $\mathbb F_q$  with $\lambda(\gamma)=\zeta_{2^m}$.  

For $3\le r\le m$, define  the integers $A_r$ and $B_r$ by
\begin{gather}
p^{s/2^{r-2}}=A_r^2+2B_r^2,\qquad A_r\equiv -1\pmod{4},\qquad p\nmid A_r,\label{eq3}\\
2B_r\equiv A_r(\gamma^{(q-1)/8}+\gamma^{3(q-1)/8})\pmod{p}.\label{eq4}
\end{gather}
It is well known that for each fixed $r$, the conditions \eqref{eq3} and \eqref{eq4} determine $A_r$ and $B_r$ uniquely.

\begin{lemma}
\label{l15}
Let $r$ be an integer with $2^{r-1}\mid s$ and $3\le r\le m$. Then
$$
G(\lambda^{2^{m-r}})+G(\bar\lambda^{2^{m-r}})=2A_r q^{(2^{r-2}-1)/2^{r-1}}
$$
and
$$
G(\lambda^{2^{m-r}})-G(\bar\lambda^{2^{m-r}})=2B_r q^{(2^{r-2}-1)/2^{r-1}}i\sqrt{2}.
$$
\end{lemma}

\begin{proof}
We observe that $\lambda^{2^{m-r}}$ has order $2^r$. By Lemma~\ref{l11}, $(\lambda^{2^{m-r}})^{2^{r-3}}=\lambda^{2^{m-3}}$ is equal to the lift of some octic character $\chi$ on $\mathbb F_{p^{s/2^{r-2}}}$ and 
$$
G(\lambda^{2^{m-r}})\pm G(\bar\lambda^{2^{m-r}})=q^{(2^{r-2}-1)/2^{r-1}}(J(\chi)\pm J(\bar\chi)).
$$
Note that $\gamma^{(q-1)/(p^{s/2^{r-2}-1})}$ is a generator of the cyclic group $\mathbb F_{p^{s/2^{r-2}}}^*$ and, by the defition of the lift, $\chi(\gamma^{(q-1)/(p^{s/2^{r-2}-1})})=\chi({\mathop{\rm N}}_{\mathbb F_q/\mathbb F_{p^{s/2^{r-2}}}}(\gamma))=\lambda^{2^{m-3}}(\gamma)=\zeta_8$. By \cite[Lemma~17]{B1}, $J(\chi)=A_r+B_ri\sqrt{2}$, and the result follows.
\end{proof}

We are now in a position to prove the main result of this section.

\begin{theorem}
\label{t1}
Let $p\equiv 3\pmod{8}$ and $m\ge 4$. Then $P_{2^m}^*(X)$  has a unique decomposition into irreducible polynomials over the rationals as follows:
\begin{itemize}
\item[\rm (a)] 
if $2^{m-1}\mid s$, then
\begin{align*}
P_{2^m}^*(X)=\,& (X-q^{\frac 12}+4B_3 q^{\frac 14})^{2^{m-2}} (X-q^{\frac 12}-4B_3 q^{\frac 14})^{2^{m-2}}\\
&\times (X-q^{\frac 12}+8B_4 q^{\frac 38})^{2^{m-3}} (X-q^{\frac 12}-8B_4 q^{\frac 38})^{2^{m-3}}\\
&\times \Bigl(X+3q^{\frac 12}-\sum_{r=3}^{m-2} 2^{r-1}A_r q^{\frac{2^{r-2}-1}{2^{r-1}}}+2^{m-2}A_{m-1} q^{\frac{2^{m-3}-1}{2^{m-2}}}\Bigr)^2\\
&\times \Bigl(X+3q^{\frac 12}-\sum_{r=3}^{m-1} 2^{r-1}A_r q^{\frac{2^{r-2}-1}{2^{r-1}}}+2^{m-1}A_m q^{\frac{2^{m-2}-1}{2^{m-1}}}\Bigr)\\
&\times \Bigl(X+3q^{\frac 12}-\sum_{r=3}^m 2^{r-1}A_r q^{\frac{2^{r-2}-1}{2^{r-1}}}\Bigr)\prod_{t=2}^{m-3}Q_t(X)^{2^{m-t-2}};
\end{align*}
\item[\rm (b)]
if $2^{m-2}\parallel s$ and $m\ge 5$, then
\begin{align*}
P_{2^m}^*(X)=\,& (X-q^{\frac 12}+4B_3 q^{\frac 14})^{2^{m-2}} (X-q^{\frac 12}-4B_3 q^{\frac 14})^{2^{m-2}}\\
&\times (X-q^{\frac 12}+8B_4 q^{\frac 38})^{2^{m-3}} (X-q^{\frac 12}-8B_4 q^{\frac 38})^{2^{m-3}}\\
&\times \Bigl(X+3q^{\frac 12}-\sum_{r=3}^{m-2} 2^{r-1}A_r q^{\frac{2^{r-2}-1}{2^{r-1}}}+2^{m-2}A_{m-1} q^{\frac{2^{m-3}-1}{2^{m-2}}}\Bigr)^2\\
&\times \left(\Bigl(X+3q^{\frac 12}-\sum_{r=3}^{m-1} 2^{r-1}A_r q^{\frac{2^{r-2}-1}{2^{r-1}}}\Bigr)^2+2^{2(m-1)}A_m^2 q^{\frac{2^{m-2}-1}{2^{m-2}}}\right)\\
&\times \left(\Bigl(X+3q^{\frac 12}-\sum_{r=3}^{m-3} 2^{r-1}A_r q^{\frac{2^{r-2}-1}{2^{r-1}}}+2^{m-3}A_{m-2} q^{\frac{2^{m-4}-1}{2^{m-3}}}\Bigr)^2\right.\\
&\hskip42pt+\Biggl.2^{2(m-1)}B_m^2 q^{\frac{2^{m-2}-1}{2^{m-2}}}\Biggr)^2\, \prod_{t=2}^{m-4}Q_t(X)^{2^{m-t-2}};
\end{align*}
\item[\rm (c)]
if $4\parallel s$, then
\begin{align*}
P_{16}^*(X)=\,&(X+3q^{\frac 12}+4A_3 q^{\frac 14})^2 (X-q^{\frac 12}+4B_3 q^{\frac 14})^4 (X-q^{\frac 12}-4B_3 q^{\frac 14})^4\\
&\times\left((X+3q^{\frac 12}-4A_3 q^{\frac 14})^2+64A_4^2 q^{\frac 34}\right)\left((X-q^{\frac 12})^2+64B_4^2 q^{\frac 34}\right)^2.
\end{align*}
\end{itemize}
The integers $A_r$ and $|B_r|$ are uniquely determined by~\eqref{eq3}, and
\begin{align*}
Q_t(X)=\,&\Bigl(X+3q^{\frac 12}-\sum_{r=3}^t 2^{r-1}A_r q^{\frac{2^{r-2}-1}{2^{r-1}}}+2^t A_{t+1} q^{\frac{2^{t-1}-1}{2^t}}+2^{t+2}B_{t+3}q^{\frac{2^{t+1}-1}{2^{t+2}}}\Bigr)\\
&\times\Bigl(X+3q^{\frac 12}-\sum_{r=3}^t 2^{r-1}A_r q^{\frac{2^{r-2}-1}{2^{r-1}}}+2^t A_{t+1} q^{\frac{2^{t-1}-1}{2^t}}-2^{t+2}B_{t+3}q^{\frac{2^{t+1}-1}{2^{t+2}}}\Bigr).
\end{align*}
\end{theorem}

\begin{proof}
Since $4\mid s$, Lemmas~\ref{l3} and \ref{l4} yield $G(\rho)=G(\lambda^{2^{m-2}})=G(\bar\lambda^{2^{m-2}})=-q^{1/2}$.  Appealing to Lemmas \ref{l14} and \ref{l15}, we deduce that
\begin{align}
\eta_0^*&=-3q^{\frac 12}+\sum\limits_{r=3}^{m-1} 2^{r-1}A_r q^{\frac{2^{r-2}-1}{2^{r-1}}}+2^{m-2}\left(G(\lambda)+G(\bar\lambda)\right),\label{eq5}\\
\eta_{2^{m-1}}^*&=-3q^{\frac 12}+\sum\limits_{r=3}^{m-1} 2^{r-1}A_r q^{\frac{2^{r-2}-1}{2^{r-1}}}-2^{m-2}\left(G(\lambda)+G(\bar\lambda)\right),\label{eq6}\\
\eta_{\pm 2^{m-2}}^*&=-3q^{\frac 12}+\sum\limits_{r=3}^{m-2} 2^{r-1}A_r q^{\frac{2^{r-2}-1}{2^{r-1}}}-2^{m-2}A_{m-1}q^{\frac{2^{m-3}-1}{2^{m-2}}},\label{eq7}\\
\eta_{\pm 2^{m-3}}^*&=\begin{cases}
q^{\frac 12}\mp 2i\sqrt{2}\left(G(\lambda)-G(\bar\lambda)\right)&\text{if $m=4$,}\\
-3q^{\frac 12}+\sum\limits_{r=3}^{m-3} 2^{r-1}A_r q^{\frac{2^{r-2}-1}{2^{r-1}}}&\\
-2^{m-3}A_{m-2}q^{\frac{2^{m-4}-1}{2^{m-3}}}\mp 2^{m-3}i\sqrt{2}\left(G(\lambda)-G(\bar\lambda)\right)&\text{if $m\ge 5$,}
\end{cases}\label{eq8}\\
\eta_{\pm 1}^*&= q^{\frac 12}\pm 4B_3 q^{\frac 14}.\label{eq9}
\end{align}
Moreover, if $m\ge 5$, then
\begin{equation}
\label{eq10}
\eta_{\pm 2}^*=q^{\frac 12}\pm 8B_4 q^{\frac 38}
\end{equation}
and, for $2\le t\le m-4$,
\begin{equation}
\eta_{\pm 2^t}^*=-3q^{\frac 12}+\sum_{r=3}^t 2^{r-1}A_r q^{\frac{2^{r-2}-1}{2^{r-1}}}-2^t A_{t+1} q^{\frac{2^{t-1}-1}{2^t}}\pm 2^{t+2}B_{t+3}q^{\frac{2^{t+1}-1}{2^{t+2}}}.\label{eq11}
\end{equation}

Assume that $2^{m-1}\mid s$. Combining \eqref{eq5}~--~\eqref{eq11} with Lemma~\ref{l15}, we obtain the values of the cyclotomic periods, which are all integers. Part~(a) now follows from Lemma~\ref{l13}.

Next assume that $2^{m-2}\parallel s$. We have $2^m\parallel(q-1)$, and so $\lambda(-1)=-1$. Hence, by 
Lemma~\ref{l2}(a),
$$
\left(G(\lambda)\pm G(\bar\lambda)\right)^2=G(\lambda)^2+G(\bar\lambda)^2\pm 2\lambda(-1)q=G(\lambda)^2+G(\bar\lambda)^2\mp 2q.
$$
Lemmas~\ref{l2}(c), \ref{l3}, \ref{l8}, \ref{l9} and \ref{l15} yield
\begin{align*}
G(\lambda)^2+G(\bar\lambda)^2&=G(\lambda)G(\lambda\rho)+G(\bar\lambda)G(\bar\lambda\rho)=\bar\lambda(4)G(\lambda^2)G(\rho)+\lambda(4)G(\bar\lambda^2)G(\rho)\\
&=-q^{1/2}(G(\lambda^2)+G(\bar\lambda^2))=-2A_{m-1}q^{(2^{m-2}-1)/2^{m-2}},
\end{align*}
and thus
\begin{equation}
\label{eq12}
\left(G(\lambda)\pm G(\bar\lambda)\right)^2=-2q^{(2^{m-2}-1)/2^{m-2}}(A_{m-1}\pm p^{s/2^{m-2}}).
\end{equation}
 Note that
$$
A_{m-1}^2+2B_{m-1}^2=p^{s/2^{m-3}}=(p^{s/2^{m-2}})^2=(A_m^2+2B_m^2)^2
=(A_m^2-2B_m^2)^2+2\cdot(2A_mB_m)^2.
$$
Hence $A_{m-1}=\pm(A_m^2-2B_m^2)$. Since $p^{s/2^{m-2}}=A_m^2+2B_m^2\equiv 3\pmod{8}$, $B_m$ is odd, and so $A_{m-1}=A_m^2-2B_m^2$. Substituting the expressions 
for $p^{s/2^{m-2}}$ and $A_{m-1}$ into \eqref{eq12}, we find that 
\begin{align*}
\left(G(\lambda)+G(\bar\lambda)\right)^2&=-4A_m^2q^{(2^{m-2}-1)/2^{m-2}},\\
\left(G(\lambda)-G(\bar\lambda)\right)^2&=8B_m^2q^{(2^{m-2}-1)/2^{m-2}}.
\end{align*}
The last two equalities together with \eqref{eq5}, \eqref{eq6} and \eqref{eq9} imply
\begin{align}
(X-\eta_0^*)(X-\eta_{2^{m-1}}^*)=\,&
\Bigl(X+3q^{\frac 12}-\sum\limits_{r=3}^{m-1} 2^{r-1}A_r q^{\frac{2^{r-2}-1}{2^{r-1}}}\Bigr)^2\notag\\
&+2^{2(m-1)}A_m^2q^{\frac{2^{m-2}-1}{2^{m-2}}},\label{eq13}\\
(X-\eta_{2^{m-3}}^*)(X-\eta_{-2^{m-3}}^*)=\,&\Bigl(X+3q^{\frac 12}-\sum\limits_{r=3}^{m-3} 2^{r-1}A_r q^{\frac{2^{r-2}-1}{2^{r-1}}}+2^{m-3}A_{m-2}q^{\frac{2^{m-4}-1}{2^{m-3}}}\Bigr)^2\notag\\
&+2^{2(m-1)}B_m^2q^{\frac{2^{m-2}-1}{2^{m-2}}}\qquad\text{if $m\ge 5$,}\label{eq14}\\
(X-\eta_{2^{m-3}}^*)(X-\eta_{-2^{m-3}}^*)=\,&(X-q^{\frac 12})^2+64B_4^2 q^{\frac 34}\qquad\text{if $m=4$.}\label{eq15}
\end{align}
Clearly, the  quadratic polynomials on the right sides of \eqref{eq13}~--~\eqref{eq15} are irreducible over the rationals.

Putting \eqref{eq7}, \eqref{eq9}~--~\eqref{eq11}, \eqref{eq13}~--~\eqref{eq15} together and appealing to Lemma~\ref{l13}, we deduce parts~(b) and (c). This completes the proof.
\end{proof}

\begin{remark}
\label{r}
{\rm The result of Gurak~\cite[Proposition~3.3(iii)]{G3} can be reformulated in terms of $A_3$ and $B_3$. Namely, $P_8^*(X)$   has the following factorization into irreducible polynomials over the rationals:
\begin{align*}
P_8^*(X)=\,& (X-q^{1/2})^2 (X-q^{1/2}+4B_3 q^{1/4})^2 (X-q^{1/2}-4B_3 q^{1/4})^2 &\\
&\times (X+3q^{1/2}+4A_3 q^{1/4})(X+3q^{1/2}-4A_3 q^{1/4})&\text{if $4\mid s$,}\\
P_8^*(X)=\,&(X-3q^{1/2})^2 &\\
&\times\left((X+q^{1/2})^2+16A_3^2 q^{1/2}\right)\left((X+q^{1/2})^2+16B_3^2 q^{1/2}\right)^2 &\text{if $2\parallel s$.}
\end{align*}
We see that Theorem~\ref{t1} is not valid for $m=3$.}
\end{remark}

\section{The Case $p\equiv 5\pmod{8}$}
\label{s4}

In this section, $p\equiv 5\pmod{8}$. As in the previous sections, $2^m\mid(q-1)$ and $\lambda$ denotes a character of order~$2^m$ on $\mathbb F_q$ such that $\lambda(\gamma)=\zeta_{2^m}$.

For $2\le r\le m-1$, define  the integers $C_r$ and $D_r$ by
\begin{gather}
p^{s/2^{r-1}}=C_r^2+D_r^2,\qquad C_r\equiv 1\pmod{4},\qquad p\nmid C_r,\label{eq16}\\
D_r \gamma^{(q-1)/4}\equiv C_r\pmod{p}.\label{eq17}
\end{gather}
If $2^{m-1}\mid s$, we extend this notation to $r=m$. It is well known that for each fixed $r$, the conditions~\eqref{eq16} and \eqref{eq17} determine $C_r$ and $D_r$ uniquely.

\begin{lemma}
\label{l16}
Let $r$ be an integer with $2^{r-1}\mid s$ and $2\le r\le m$. Then
$$
G(\lambda^{2^{m-r}})+G(\bar\lambda^{2^{m-r}})=\begin{cases}
-2C_r q^{(2^{r-1}-1)/2^r}&\text{if $2^r\mid s$,}\\
(-1)^r\cdot 2C_r q^{(2^{r-1}-1)/2^r}&\text{if $2^{r-1}\parallel s$,}
\end{cases}
$$
and
$$
G(\lambda^{2^{m-r}})-G(\bar\lambda^{2^{m-r}})=\begin{cases}
2D_r q^{(2^{r-1}-1)/2^r}i&\text{if $2^r\mid s$,}\\
(-1)^{r-1}\cdot 2D_r q^{(2^{r-1}-1)/2^r}i&\text{if $2^{r-1}\parallel s$.}
\end{cases}
$$
\end{lemma}

\begin{proof}
The proof proceeds exactly as for Lemma~\ref{l15}, except that at the end, \cite[Proposition~3]{KR} is invoked instead of \cite[Lemma~17]{B1}.
\end{proof}

We are now ready to establish our second main result.

\begin{theorem}
\label{t2}
Let $p\equiv 5\pmod{8}$ and $m\ge 4$. Then $P_{2^m}^*(X)$   has a unique decomposition into irreducible polynomials over the rationals as follows:
\begin{itemize}
\item[\rm (a)]
if $2^m\mid s$, then
\begin{align*}
P_{2^m}^*(X)=\,& (X-q^{\frac 12}+2D_2 q^{\frac 14})^{2^{m-2}} (X-q^{\frac 12}-2D_2 q^{\frac 14})^{2^{m-2}}\\
&\times\Bigl(X+q^{\frac 12}+\sum_{r=2}^{m-1}2^{r-1}C_r q^{\frac{2^{r-1}-1}{2^r}}-2^{m-1}C_m q^{\frac{2^{m-1}-1}{2^m}}\Bigr)\\
&\times\Bigl(X+q^{\frac 12}+\sum_{r=2}^m 2^{r-1}C_r q^{\frac{2^{r-1}-1}{2^r}}\Bigr)\prod_{t=1}^{m-2}R_t(X)^{2^{m-t-2}};
\end{align*}
\item[\rm (b)]
if $2^{m-1}\parallel s$, then
\begin{align*}
P_{2^m}^*(X)=\,& (X-q^{\frac 12}+2D_2 q^{\frac 14})^{2^{m-2}} (X-q^{\frac 12}-2D_2 q^{\frac 14})^{2^{m-2}}\\
&\times\left(\Bigl(X+q^{\frac 12}+\sum_{r=2}^{m-1}2^{r-1}C_r q^{\frac{2^{r-1}-1}{2^r}}\Bigl)^2 -2^{2(m-1)}C_m^2 q^{\frac{2^{m-1}-1}{2^{m-1}}}\right)\\
&\times\Biggl(\Bigl(X+q^{\frac 12}+\sum_{r=2}^{m-2}2^{r-1}C_r q^{\frac{2^{r-1}-1}{2^r}}-2^{m-2}C_{m-1}q^{\frac{2^{m-2}-1}{2^{m-1}}}\Bigr)^2\Biggr.\\
&\hskip20pt -\Biggl.2^{2(m-1)}D_m^2 q^{\frac{2^{m-1}-1}{2^{m-1}}}\Biggr)\prod_{t=1}^{m-3}R_t(X)^{2^{m-t-2}};
\end{align*}
\item[\rm (c)]
if $2^{m-2}\parallel s$, then
\begin{align*}
P_{2^m}^*(X)=\,& (X-q^{\frac 12}+2D_2 q^{\frac 14})^{2^{m-2}} (X-q^{\frac 12}-2D_2 q^{\frac 14})^{2^{m-2}}\\
&\times\biggl(\Bigl(X+q^{\frac 12}+\sum_{r=2}^{m-3}2^{r-1}C_r q^{\frac{2^{r-1}-1}{2^r}}-2^{m-3}C_{m-2}q^{\frac{2^{m-3}-1}{2^{m-2}}}\Bigr)^2\biggr.\\
&\hskip26pt-\biggl.2^{2(m-2)}D_{m-1}^2 q^{\frac{2^{m-2}-1}{2^{m-2}}}\biggr)^2
\end{align*}
\begin{align*}
&\times\Biggl(\biggl(\Bigl(X+q^{\frac 12}+\sum_{r=2}^{m-2}2^{r-1}C_r q^{\frac{2^{r-1}-1}{2^r}}\Bigr)^2+2^{2(m-2)}C_{m-1}^2 q^{\frac{2^{m-2}-1}{2^{m-2}}}+2^{2m-3}q\biggr)^2\Biggr.\\
&\hskip26pt -2^{2(m-1)}C_{m-1}^2 q^{\frac{2^{m-2}-1}{2^{m-2}}}\Biggl.\Bigl(X+(2^{m-2}+1)q^{\frac 12}+\sum_{r=2}^{m-2}2^{r-1}C_r q^{\frac{2^{r-1}-1}{2^r}}\Bigr)^2\Biggr)\\
&\times\prod_{t=1}^{m-4}R_t(X)^{2^{m-t-2}}.
\end{align*}
\end{itemize}
The integers $C_r$ and $|D_r|$ are uniquely determined by~\eqref{eq16}, and
\begin{align*}
R_t(X)=\,&\Bigl(X+q^{\frac 12}+\sum_{r=2}^t 2^{r-1}C_r q^{\frac{2^{r-1}-1}{2^r}}-2^t C_{t+1} q^{\frac{2^t-1}{2^{t+1}}}+2^{t+1} D_{t+2} q^{\frac{2^{t+1}-1}{2^{t+2}}}\Bigr)\\
&\times\Bigl(X+q^{\frac 12}+\sum_{r=2}^t 2^{r-1}C_r q^{\frac{2^{r-1}-1}{2^r}}-2^t C_{t+1} q^{\frac{2^t-1}{2^{t+1}}}-2^{t+1} D_{t+2} q^{\frac{2^{t+1}-1}{2^{t+2}}}\Bigr).
\end{align*}
\end{theorem}

\begin{proof}
As $s$ is even, Lemma~\ref{l3} yields $G(\rho)=-q^{1/2}$. Then, applying Lemmas~\ref{l14} and \ref{l16}, we obtain 
\begin{align}
\eta_0^*=\,&-q^{\frac 12}-\sum_{r=2}^{m-2}2^{r-1}C_r q^{\frac{2^{r-1}-1}{2^r}}+2^{m-3}\left(G(\lambda^2)+G(\bar\lambda^2)\right)\notag\\
&+2^{m-2}\left(G(\lambda)+G(\bar\lambda)\right),\label{eq18}\\
\eta_{2^{m-1}}^*=\,&-q^{\frac 12}-\sum_{r=2}^{m-2}2^{r-1}C_r q^{\frac{2^{r-1}-1}{2^r}}+2^{m-3}\left(G(\lambda^2)+G(\bar\lambda^2)\right)\notag\\
&-2^{m-2}\left(G(\lambda)+G(\bar\lambda)\right),\label{eq19}\\
\eta_{\pm 2^{m-2}}^*=\,&-q^{\frac 12}-\sum_{r=2}^{m-2}2^{r-1}C_r q^{\frac{2^{r-1}-1}{2^r}}-2^{m-3}\left(G(\lambda^2)+G(\bar\lambda^2)\right)\notag\\
&\mp 2^{m-2}i\left(G(\lambda)-G(\bar\lambda)\right),\label{eq20}\\
\eta_{\pm 2^{m-3}}^*=\,&-q^{\frac 12}-\sum_{r=2}^{m-3}2^{r-1}C_r q^{\frac{2^{r-1}-1}{2^r}}+2^{m-3}C_{m-2}q^{\frac{2^{m-3}-1}{2^{m-2}}}\notag\\
&\mp 2^{m-3}i\left(G(\lambda^2)-G(\bar\lambda^2)\right),\label{eq21}\\
\eta_{\pm 1}^*=\,&q^{\frac 12}\pm 2D_2 q^{\frac 14},\label{eq22}
\end{align}
and, for $1\le t\le m-4$,
\begin{equation}
\label{eq23}
\eta_{\pm 2^t}^*=-q^{\frac 12}-\sum_{r=2}^t 2^{r-1}C_r q^{\frac{2^{r-1}-1}{2^r}}+2^t C_{t+1} q^{\frac{2^t-1}{2^{t+1}}}\pm 2^{t+1} D_{t+2} q^{\frac{2^{t+1}-1}{2^{t+2}}}.
\end{equation}

First suppose that $2^m\mid s$. By combining \eqref{eq18}~--~\eqref{eq23} with Lemma~\ref{l16}, we find the values of the cyclotomic periods, which are all integers. Now part~(a) follows from Lemma~\ref{l13}.

Next suppose that $2^{m-1}\parallel s$.  Using \eqref{eq18}~--~\eqref{eq23} and Lemma~\ref{l16} again, we find the values of the cyclotomic periods. We observe that $\eta_0^*$ and $\eta_{2^{m-1}}^*$ as well as $\eta_{2^{m-2}}^*$ and $\eta_{-2^{m-2}}^*$ are algebraic conjugates of degree~2 over the rationals, and the remaining cyclotomic periods are integers. Therefore the polynomials
$$
(X-\eta_0^*)(X-\eta_{2^{m-1}}^*)=\Bigl(X+q^{\frac 12}+\sum_{r=2}^{m-1}2^{r-1}C_r q^{\frac{2^{r-1}-1}{2^r}}\Bigr)^2-2^{2(m-1)}C_m^2 q^{\frac{2^{m-1}-1}{2^{m-1}}}
$$
and
\begin{align*}
(X-\eta_{2^{m-2}}^*)(X-\eta_{-2^{m-2}}^*)=\,&\Bigl(X+q^{\frac 12}+\sum_{r=2}^{m-2}2^{r-1}C_r q^{\frac{2^{r-1}-1}{2^r}}-2^{m-2}C_{m-1}q^{\frac{2^{m-2}-1}{2^{m-1}}}\Bigr)^2\\
&-2^{2(m-1)}D_m^2 q^{\frac{2^{m-1}-1}{2^{m-1}}}
\end{align*}
are irreducible over the rationals. Part~(b) now follows in view of  Lemma~\ref{l13}.

Finally, suppose that $2^{m-2}\parallel s$. Making use of \eqref{eq18}~--~\eqref{eq20} and Lemma~\ref{l16}, we obtain
\begin{align}
\eta_0^*=\,&-q^{\frac 12}-\sum_{r=2}^{m-2}2^{r-1}C_r q^{\frac{2^{r-1}-1}{2^r}}-(-1)^m\cdot 2^{m-2}C_{m-1}q^{\frac{2^{m-2}-1}{2^{m-1}}}\notag\\
&+2^{m-2}\left(G(\lambda)+G(\bar\lambda)\right),\label{eq24}\\
\eta_{2^{m-1}}^*=\,&-q^{\frac 12}-\sum_{r=2}^{m-2}2^{r-1}C_r q^{\frac{2^{r-1}-1}{2^r}}-(-1)^m\cdot 2^{m-2}C_{m-1}q^{\frac{2^{m-2}-1}{2^{m-1}}}\notag\\
&-2^{m-2}\left(G(\lambda)+G(\bar\lambda)\right),\label{eq25}\\
\eta_{\pm 2^{m-2}}^*=\,&-q^{\frac 12}-\sum_{r=2}^{m-2}2^{r-1}C_r q^{\frac{2^{r-1}-1}{2^r}}+(-1)^m\cdot 2^{m-2}C_{m-1}q^{\frac{2^{m-2}-1}{2^{m-1}}}\notag\\
&\mp 2^{m-2}i\left(G(\lambda)-G(\bar\lambda)\right).\label{eq26}
\end{align}
By employing the same type of argument as in the proof of Theorem~\ref{t1}, we see that
$$
\left(G(\lambda)\pm G(\bar\lambda)\right)^2=\mp\, 2q^{(2^{m-1}-1)/2^{m-1}}\left(q^{1/2^{m-1}}\pm (-1)^m\cdot C_{m-1}\right).
$$
Combining this with \eqref{eq24}~--~\eqref{eq26}, we conclude that
\begin{align*}
(X-\eta_0^*)(X-&\eta_{2^{m-1}}^*)\\
=\,&\Bigl(X+q^{\frac 12}+\sum_{r=2}^{m-2}2^{r-1}C_r q^{\frac{2^{r-1}-1}{2^r}}+(-1)^m\cdot 2^{m-2}C_{m-1}q^{\frac{2^{m-2}-1}{2^{m-1}}}\Bigr)^2\\
&+2^{2m-3}q^{\frac{2^{m-1}-1}{2^{m-1}}}\left(q^{\frac 1{2^{m-1}}}+(-1)^m\cdot C_{m-1}\right)
\end{align*}
and
\begin{align*}
(X-\eta_{2^{m-2}}^*)(X-&\eta_{-2^{m-2}}^*)\\
=\,&\Bigl(X+q^{\frac 12}+\sum_{r=2}^{m-2}2^{r-1}C_r q^{\frac{2^{r-1}-1}{2^r}}-(-1)^m\cdot 2^{m-2}C_{m-1}q^{\frac{2^{m-2}-1}{2^{m-1}}}\Bigr)^2\\
&+2^{2m-3}q^{\frac{2^{m-1}-1}{2^{m-1}}}\left(q^{\frac 1{2^{m-1}}}-(-1)^m\cdot C_{m-1}\right).
\end{align*}
 Since $q^{1/2^{m-2}}=p^{s/2^{m-2}}=C_{m-1}^2+D_{m-1}^2$, we have $q^{1/2^{m-1}}>|C_{m-1}|$. This means that the polynomials $(X-\eta_0^*)(X-\eta_{2^{m-1}}^*)$ and $(X-\eta_{2^{m-2}}^*)(X-\eta_{-2^{m-2}}^*)$ are irreducible over the reals. Furthermore, since $2^{m-2}\parallel s$, the polynomials $(X-\eta_0^*)(X-\eta_{2^{m-1}}^*)$ and $(X-\eta_{2^{m-2}}^*)(X-\eta_{-2^{m-2}}^*)$ belong to $\mathbb R[X]\setminus\mathbb Q[X]$. Since $\mathbb R[X]$ is a unique factorization domain, it follows that the polynomial
\begin{align*}
(X-&\eta_0^*)(X-\eta_{2^{m-1}}^*)(X-\eta_{2^{m-2}}^*)(X-\eta_{-2^{m-2}}^*)\\
=\,&\biggl(\Bigl(X+q^{\frac 12}+\sum_{r=2}^{m-2}2^{r-1}C_r q^{\frac{2^{r-1}-1}{2^r}}\Bigr)^2+2^{2(m-2)}C_{m-1}^2 q^{\frac{2^{m-2}-1}{2^{m-2}}}+2^{2m-3}q\biggr)^2\\
&-2^{2(m-1)}C_{m-1}^2 q^{\frac{2^{m-2}-1}{2^{m-2}}}\Bigl(X+(2^{m-2}+1)q^{\frac 12}+\sum_{r=2}^{m-2}2^{r-1}C_r q^{\frac{2^{r-1}-1}{2^r}}\Bigr)^2
\end{align*}
is irreducible over the rationals. Further, by Lemma~\ref{l16} and \eqref{eq21},
$$
\eta_{\pm 2^{m-3}}^*=-q^{\frac 12}-\sum_{r=2}^{m-3}2^{r-1}C_r q^{\frac{2^{r-1}-1}{2^r}}+ 2^{m-3}C_{m-2}q^{\frac{2^{m-3}-1}{2^{m-2}}}\pm (-1)^m\cdot 2^{m-2}D_{m-1}q^{\frac{2^{m-2}-1}{2^{m-1}}},
$$
and so $\eta_{2^{m-3}}^*$ and $\eta_{-2^{m-3}}^*$ are algebraic conjugates of degree~2 over the rationals. Hence, the polynomial
\begin{align*}
(X-\eta_{2^{m-3}}^*)(X-\eta_{-2^{m-3}}^*)=\,&\Bigl(X+q^{\frac 12}+\sum_{r=2}^{m-3}2^{r-1}C_r q^{\frac{2^{r-1}-1}{2^r}}-2^{m-3}C_{m-2}q^{\frac{2^{m-3}-1}{2^{m-2}}}\Bigr)^2\\
&-2^{2(m-2)}D_{m-1}^2 q^{\frac{2^{m-2}-1}{2^{m-2}}}
\end{align*}
is irreducible over the rationals. The remaining cyclotomic periods $\eta_{\pm 2^t}^*$, $0\le t\le m-4$, are integers, in view of \eqref{eq22} and \eqref{eq23}. Now Part~(c) follows by appealing to Lemma~\ref{l13}. This concludes the proof.
\end{proof}

\begin{remark}
{\rm Myerson has shown \cite[Theorem~17]{M} that $P_4^*(X)$ is irreducible if $2\nmid s$,
\begin{align*}
P_4^*(X)=\,& (X+q^{1/2}+2C_2 q^{1/4})(X+q^{1/2}-2C_2 q^{1/4})&\\
&\times(X-q^{1/2}+2D_2 q^{1/4}) (X-q^{1/2}-2D_2 q^{1/4})&\text{if $4\mid s$,}\\
\intertext{and, with a slight modification,}
P_4^*(X)=\,&\left((X+q^{1/2})^2-4C_2^2q^{1/2}\right)\left((X-q^{1/2})^2-4D_2^2q^{1/2}\right)&\text{if $2\parallel s$,}
\end{align*}
where in the latter case the quadratic polynomials are irreducible over the rationals. Furthermore, the result of Gurak~\cite[Proposition~3.3(ii)]{G3} can be reformulated in terms of $C_2$, $D_2$, $C_3$ and $D_3$. Namely, $P_8^*(X)$   has the following factorization into irreducible polynomials over the rationals:
\begin{align*}
P_8^*(X)=\,& (X-q^{1/2}+2D_2 q^{1/4})^2 (X-q^{1/2}-2D_2 q^{1/4})^2\\
&\times(X+q^{1/2}+2C_2 q^{1/4}+4C_3 q^{3/8})(X+q^{1/2}+2C_2 q^{1/4}-4C_3 q^{3/8})&\\
&\times(X+q^{1/2}-2C_2 q^{1/4}+4D_3 q^{3/8})(X+q^{1/2}-2C_2 q^{1/4}-4D_3 q^{3/8})&\!\text{if $8\mid s$,}\\
P_8^*(X)=\,& (X-q^{1/2}+2D_2 q^{1/4})^2 (X-q^{1/2}-2D_2 q^{1/4})^2&\\
&\times\left((X+q^{1/2}+2C_2 q^{1/4})^2-16C_3^2 q^{3/4}\right)&\\
&\times\left((X+q^{1/2}-2C_2 q^{1/4})^2-16D_3^2 q^{3/4}\right)&\!\text{if $4\parallel s$,}\\
P_8^*(X)=\,&\left((X-q^{1/2})^2-4D_2^2 q^{1/2}\right)^2&\\
&\times\left(\left((X+q^{1/2})^2+4C_2^2 q^{1/2}+8q\right)^2-16C_2^2 q^{1/2}(X+3q^{1/2})^2\right)&\!\text{if $2\parallel s$.}
\end{align*}
Thus part~(a) of Theorem~\ref{t2} remains valid for $m=2$ and $m=3$. Moreover, for $m=3$, part~(b) of Theorem~\ref{t2} is still valid (cf. Remark~\ref{r}).}
\end{remark}


\begin{thebibliography}{00}
\bibitem{B1}
I.~Baoulina, Generalizations of the Markoff-Hurwitz equations over finite fields, {\it J.~Number Theory} {\bf 118}~(2006) 31--52.

\bibitem{B2}
I.~Baoulina, On the number of solutions to the equation $(x_1+\cdots +x_n)^2=ax_1\cdots x_n$ in a finite field, {\it Int. J. Number Theory} {\bf 4}~(2008) 797--817.

\bibitem{BM}
L.~D.~Baumert and J.~Mykkeltveit, Weight distributions of some irreducible cyclic codes, {\it DSN Progr. Rep.} {\bf 16}~(1973) 128--131.

\bibitem{BEW}
B.~C.~Berndt, R.~J.~Evans and K.~S.~Williams, {\it Gauss and Jacobi Sums}
(Wiley-Interscience, New York, 1998).

\bibitem{G1}
S.~Gurak, Factors of period polynomials for finite fields, I, in {\it The Rademacher Legacy to Mathematics}, eds. G.~E.~Andrews, D.~M.~Bressoud and L.~A.~Parson, Contemp. Math., Vol.~166 (American Mathematical Society, 1994), pp.~309--333.

\bibitem{G2}
S.~Gurak, Period polynomials for $\mathbb F_{p^2}$ of fixed small degree, in {\it Finite Fields and Applications}, eds. D.~Jungnickel and H.~Niederreiter (Springer Berlin Heidelberg, 2000), pp.~196--207.

\bibitem{G3}
S.~J.~Gurak, Period polynomials for $\mathbb F_q$ of fixed small degree, in {\it Number Theory}, eds. H.~Kisilevsky and E.~Z.~Goren, CRM
Proc. and Lect. Notes, Vol.~36 (American Mathematical Society, 2004), pp.~127--145.

\bibitem{KR}
S.~A.~Katre and A.~R.~Rajwade, Resolution of the sign ambiguity in
the determination of the cyclotomic numbers of order $4$ and the
corresponding Jacobsthal sum, {\it Math. Scand.} {\bf 60}~(1987) 52--62.

\bibitem{LN}
R.~Lidl and H.~Niederreiter, {\it Finite Fields} (Addison-Wesley, Reading, MA, 1983).

\bibitem{M}
G.~Myerson, Period polynomials and Gauss sums for finite fields, {\it Acta Arith.} {\bf 39}~(1981) 251--264.


\end{thebibliography}
\end{document}